\documentclass{amsart}
\usepackage[latin1]{inputenc}
\usepackage[english]{babel}

\usepackage{color}

\usepackage{indentfirst}
\usepackage{amssymb}
\usepackage{amsthm}
\usepackage{amscd}

\def\natnums{\mathbb N}
\def\reals{\mathbb R}
\def\R{\reals}

\def\N{\natnums}

\newtheorem{theo}{Theorem}
\newtheorem{lem}[theo]{Lemma}
\newtheorem{cor}[theo]{Corollary}

\newtheorem{thm}[theo]{Theorem}
\newtheorem{prop}[theo]{Proposition}
\newtheorem{defn}[theo]{Definition}

\theoremstyle{definition}

\theoremstyle{remark}

\numberwithin{equation}{section}

\def\epsilon{\varepsilon}

\newcommand{\To}{\longrightarrow}

\providecommand{\MR}{\relax\ifhmode\unskip\space\fi MR }

\providecommand{\href}[2]{#2}

\title[A 1-separably injective space with without $\ell_\infty$]{A 1-separably injective Banach space that does not contain $\ell_\infty$}

\author{ Antonio Avil\'{e}s}
\address{Departamento de Matem\'{a}ticas\\
Facultad de Matem\'{a}ticas\\ Universidad de Murcia\\ 30100 Espinardo (Murcia)\\
Spain} \email{avileslo@um.es}

\author{Piotr Koszmider}
\address{Institute of Mathematics\\ Polish Academy of Sciences \\ \'{S}niadeckich 8\\ 00-656 Warsaw \\ Poland} \email{piotr.koszmider@impan.pl}



\thanks{First author supported by MINECO and FEDER (MTM2014-54182-P) and by Fundaci\'{o}n S\'{e}neca - Regi\'{o}n de Murcia (19275/PI/14). The research  of the second named author was partially supported by   grant PVE Ci\^{e}ncia sem Fronteiras - CNPq (406239/2013-4). This work was done during a visit of the first author to the IMPAN in Warsaw and another visit of both authors to the University of S\~{a}o Paulo.}

\begin{document}

\begin{abstract}
We show that the problem whether every $1$-separably injective Banach space contains
an isomorphic copy of $\ell_\infty$ is undecidable. 
Namely, unlike under the continuum hypothesis, assuming Martin's axiom and the negation of the continuum hypothesis, there is an $1$-separably injective Banach space of the form $C(K)$ (which means that $K$
is an $F$-space) without an isomorphic copy of $\ell_\infty$. 

This result is a consequence of our  study of $\omega_2$-subsets of tightly $\sigma$-filtered Boolean algebras
introduced by Koppelberg  for which we obtain some general principles 
useful when transferring  properties of Boolean algebras to the level of Banach spaces.
\end{abstract}

\maketitle

\section{Introduction}

A Banach space $\mathcal X$ is called $\lambda$-injective if given two Banach
spaces $\mathcal Y\subseteq \mathcal Z$ and a bounded linear
operator $T: \mathcal Y\rightarrow \mathcal X$ 
there is a bounded linear operator $\widetilde{T}: \mathcal Z\rightarrow \mathcal X$ which
extends $T$ and satisfies $\|\widetilde{T}\|\leq \lambda \|T\|$. If $\mathcal X$ is $\lambda$-injective
for some $\lambda\in \R$, then we say that $\mathcal X$ is injective. 
Despite major efforts since the 50ties it is not know what are exactly injective Banach spaces.
Nachbin and Kelley proved in \cite{nachbin} and \cite{kelley} that $1$-injective Banach spaces are isometric
to the spaces of the form $C(\mathfrak K_\mathfrak B)$ for some injective Boolean algebra $\mathfrak B$,
that is, for a complete Boolean algebra $\mathfrak B$, where $\mathfrak K_\mathfrak B$ is the Stone space
of $\mathfrak B$, which is equivalent for a compact space to being extremally disconnected. 
Rosenthal in his attempt of understanding injective Banach spaces (\cite{rosenthal1, rosenthal2}) 
proved, among others, that
all injective Banach spaces contain isomorphic copies of $\ell_\infty$.

One can consider  natural separable versions of $\lambda$-injectivity. Namely,
a Banach space $\mathcal X$ is called  $\lambda$-separably injective if given two separable Banach
spaces $\mathcal Y\subseteq \mathcal Z$ and a bounded linear
operator $T: \mathcal Y\rightarrow \mathcal X$ 
there is a bounded linear operator $\widetilde{T}: \mathcal Z\rightarrow \mathcal X$ which
extends $T$ and satisfies $\|\widetilde{T}\|\leq \lambda \|T\|$. If $\mathcal X$ is $\lambda$-separably injective
for some $\lambda\in \R$, then we say that $\mathcal X$ is separably injective.  In these terms
the classical Zippin's theorem says that every separable and separably injective Banach space
is isomorphic to $c_0$. However,  Ostrovskii essentially proved (\cite{ostrovskii}, cf. \cite{sepinybook}) that there are
no separable  $\lambda$-separably injective infinite dimensional Banach spaces
for $\lambda<2$. Nonseparable $\lambda$-separably injective Banach spaces have been 
investigated in the book \cite{sepinybook}, where among others many examples of nonseparable
separably injective Banach spaces not containing an isomorphic copy of $\ell_\infty$ are given.

However, assuming the Continuum Hypothesis  all $1$-separably
injective Banach spaces contain an isomorphic copy of $\ell_\infty$. This follows from
a result of Lindenstrauss from  (\cite{lindenstrauss})
that under this hypothesis all $1$-separably injective spaces are universally separably
injective, that is satisfy a version of the definition of separably injective Banach space where
only $\mathcal Y$ must be separable and $\mathcal Z$ has arbitrary density.  Having this
fact one uses an observation that always all universally separably injective Banach spaces contain
an isomorphic copy of $\ell_\infty$ (\cite{sepinybook}). It should be added that without any special set-theoretic assumptions all
$1$-separably injective spaces are Grothendieck and Lindenstrauss.

The main result of this paper is to show that some extra set-theoretic assumption 
is necessary to obtain such isomorphic copies of $\ell_\infty$ as we show that it is
consistent that there is a $1$-separably injective Banach space without an isomorphic copy 
of $\ell_\infty$ (Corollary \ref{main-result}). Our Banach space is of the form $C(\mathfrak K)$ for 
$\mathfrak K$ Hausdorff and compact.
Such $1$-separably injective spaces  are 
exactly those, where  $\mathfrak K$ is a so called $F$-space, that is compact Hausdorff,  where closures
of disjoint $F_\sigma$-sets are disjoint (a little bit more generally, a $C^*$-algebra
is a $1$-separably injective Banach space if and only if it is commutative and of the form
$C_0(X)$ where $X$ is locally compact and substonean \cite{chu}). So
our result may also be considered as a strengthening to
the level of Banach spaces of a result of Dow and Hart (\cite{DowHart})
that it is consistent that there is an $F$-space which does not continuously map onto
$\beta\N$.  

In fact, the main result follows from our analysis of  infinite Boolean algebras  $\mathfrak{B}$ for which 
one can consider  the following extension properties:
\begin{enumerate}
\item (countable separation property) For every two countable sets $X,Y\subset \mathfrak{B}$ such that $x\leq y$ for all $x\in X$ and $y\in Y$, there exists $z\in \mathfrak{B}$ such that $x\leq z \leq y$ for all $x\in X$, $y\in Y$.
\item If $A_0\subset A_1$ are Boolean algebras with $A_1$ generated by a single element over $A_0$, then every Boolean morphism $A_0\To \mathfrak{B}$ extends to $A_1$.
\item If $A_0\subset A_1$ are Boolean algebras with $A_1$ countable, then every Boolean morphism $A_0\To \mathfrak{B}$ extends to $A_1$.   
\item If $A_0\subset A_1$ are Boolean algebras with $A_0$ countable, then every Boolean morphism $A_0\To \mathfrak{B}$ extends to $A_1$.
\item If $A_0\subset A_1$ are Boolean algebras with $A_0$ countable and $A_1$ isomorphic to $\mathcal{P}(\omega)$, then every Boolean morphism $A_0\To \mathfrak{B}$ extends to $A_1$.
\item $\mathfrak{B}$ contains a subalgebra isomorphic to $\mathcal{P}(\omega)$.
\end{enumerate}

The implications $1 \Leftrightarrow 2 \Leftrightarrow 3 \Leftarrow 4 \Leftrightarrow 5 \Rightarrow 6$ hold in ZFC, while $3\Rightarrow 4$ holds under CH. However, Dow and Hart \cite[Theorem 5.10]{DowHart} proved that the implication $3\Rightarrow 6$ fails under MA and $\mathfrak{c}=\aleph_2$. The Stone space of a Boolean algebra with property
1 is an F-space, so, in topological terms, as mentioned above, what Dow and Hart found is a zero-dimensional F-space that does not map continuously onto $\beta\omega$. Analogous properties can be considered for a Banach space $\mathfrak{X}$, cf. \cite{sepiny} which using our previous terminology have the form:

\begin{enumerate}
\item If we have a countable family of balls such that every two of them have nonempty intersection, then the intersection of the whole family is nonempty.
\item If $A_0\subset A_1$ are Banach spaces with $A_1$ generated by a single element over $A_0$, then every bounded operator $A_0\To \mathfrak{X}$ extends to a bounded operator on $A_1$ with the same norm.
\item (1-separably injective) If $A_0\subset A_1$ are Banach spaces with $A_1$ separable, then every bounded operator $A_0\To \mathfrak{X}$ extends to a bounded operator on $A_1$ with the same norm.
\item (1-universally separably injective) If $A_0\subset A_1$ are Banach spaces with $A_0$ separable, then every bounded operator $A_0\To \mathfrak{X}$ extends to a bounded operator on $A_1$ with the same norm.
\item If $A_0\subset A_1$ are Banach spaces with $A_0$ separable and $A_1$ isomorphic to $\ell_\infty$, then every bounded operator $A_0\To \mathfrak{X}$ extends to a bounded operator on $A_1$ with the same norm.
\item $\mathfrak{X}$ contains a subspace isomorphic to $\ell_\infty$.
\end{enumerate}

Again, $1 \Leftrightarrow 2 \Leftrightarrow 3 \Leftarrow 4 \Leftrightarrow 5 \Rightarrow 6$ hold in ZFC, while $3\Rightarrow 4$ holds under CH by the previous discussion. It is proven in \cite{sepiny} that if $\mathfrak{K}$ is the Stone space of the Boolean algebra of Dow and Hart mentioned above, then the space of continuous functions $C(\mathfrak{K})$ shows that $3\not\Rightarrow 4$ under MA and $\mathfrak{c}=\aleph_2$. However it was left as an open problem whether one could consistently disprove $3\Rightarrow 6$. Our main result is an answer to this question, indeed we are able to prove that the same space $C(\mathfrak{K})$ does not contain isomorphic copies of $\ell_\infty$ under MA and $\mathfrak{c}>\aleph_1$. The Banach spaces of universal dispostion $\mathfrak{S}^{\omega_1}(X)$ constructed in \cite{unidisp} starting from a separable space $X$, are subspaces of $C(\mathfrak{K})$ by combined results of \cite[Theorem 29]{AviBre} and \cite{sepinybook}. Hence these spaces consistently do not contain isomorphic copies of $\ell_\infty$.

The key notion of this paper is that of a tightly $\sigma$-filtered Boolean algebra, studied in \cite{Koppelberg,Geschke,AviBre} and under different names in \cite{DowHart}. We generalize the fact that tightly $\sigma$-filtered algebras do not contain $\omega_2$-chains \cite[Proposition 2.5]{DowHart} to a more general principle (Theorem~\ref{omegathm}) asserting that if an $\omega_2$-sequence of elements often satisfies some Boolean equation, then this equation must be also often satisfied in permuted order. We further export these ideas to the computation of norms in the space of continuous functions on the Stone space, and then we use a trick from \cite{BrechKoszmider}  to relate the embeddability of $\ell_\infty$ inside $C(\mathfrak{K})$ to the existence of $\omega_2$-chains in $\mathcal{P}(\omega)/fin$.

\section{Tightly $\sigma$-filtered Boolean algebras}

 Two subalgebras $A_1$ and $A_2$ of a Boolean algebra $B$ are said to commute if whenever we have $a_1\wedge a_2 = 0$ with $a_i\in A_i$, then there exist $b_1,b_2\in A_1\cap A_2$ such that $a_1\leq b_1$, $a_2\leq b_2$ and $b_1\wedge b_2 = 0$. This is equivalent to say that for every $c_1\in A_1$ and $c_2\in A_2$ with $c_1\leq c_2$ there exists $r\in R$ such that $c_1\leq r \leq c_2$. A diagram of injective Boolean morphisms $$
\begin{CD}
A @>>> B\\
@AAA @AAA\\
R @>>> S
\end{CD}
$$
is a push-out diagram if, when we identify all these algebras as subalgebras of $B$ through those morphisms, we have that $\langle A\cup S\rangle = B$ (we denote by $\langle X \rangle$ the subalgebra generated by the set $X$), $A\cap S = R$, and $A$ and $S$ commute. This is equivalent to the fact that this is a push-out diagram in the category of Boolean algebras and (not necessarily injective) morphisms, in the general categorical sense. From now on, all arrows between Boolean algebras are supposed to be injective morphisms. Whenever we are given arrows
$$
\begin{CD}
A \\
@AAA \\
R @>>> S
\end{CD}
$$
there is a unique way to complete the diagram to obtain a push-out diagram. This is to make the free sum of $S$ and $A$ and make a quotient to identify the image of $R$ in $A$ and the image of $R$ in $S$.

\begin{defn}
A Boolean algebra $\mathfrak{B}$ is tightly $\sigma$-filtered if there exists a chain of subalgebras $\{B_\alpha\}_{\alpha<\xi}$ indicated in an ordinal $\xi$ with the following properties:
\begin{enumerate}
\item $B_0 = \{0,1\}$,
\item $\mathfrak{B} = \bigcup_{\alpha<\xi}B_\alpha$,
\item $B_\alpha = \bigcup_{\beta<\alpha} B_{\beta}$ if $\alpha$ is a limit ordinal
\item For every $\alpha<\xi$ there exist countable subalgebras $R_\alpha,S_\alpha\subset \mathfrak{B}$ such that the diagram of inclusions
$$
\begin{CD}
B_{\alpha} @>>> B_{\alpha+1}\\
@AAA @AAA\\
R_\alpha @>>> S_\alpha
\end{CD}
$$
is a push-out diagram.
\end{enumerate}
\end{defn}

It is easy to construct tightly $\sigma$-filtered Boolean algebras by induction. One starts with $\{0,1\}$, at limit steps one takes unions, and at successor stages one chooses any countable subalgebra $R_\alpha\subset B_\alpha$ and an external countable superalgebra $S_\alpha\supset R_\alpha$ and then one takes $B_{\alpha+1}$ to be the push-out of the diagram
$$
\begin{CD}
B_{\alpha} \\
@AAA \\
R_\alpha @>>> S_\alpha
\end{CD}
$$
For example, free Boolean algebras are tightly $\sigma$-filtered. They are the one that we obtain if in this procedure we take $R_\alpha = \{0,1\}$ and $S_\alpha = \{0,1,a,b\}$ all the time. Another example, more relevant for us, is the following: If $B$ has cardinality at most $\mathfrak{c}$ then there are, up to isomorphism, only $\mathfrak{c}$ diagrams of the form 
$$
\begin{CD}
B \\
@AAA \\
R @>>> S
\end{CD}
$$
with $R$, $S$ countable algebras. Therefore, by a book-keeping argument, we can construct a tightly $\sigma$-filtered Boolean algebra $\mathfrak{B} = \bigcup_{\alpha<\mathfrak{c}} B_\alpha$ such that, for every diagram of the form 
$$
\begin{CD}
B_\alpha \\
@AAA \\
R @>>> S
\end{CD}
$$
with $R$, $S$ countable, there exists $\beta>\alpha$ such that $R_\beta = R$ and we have naturally identified diagrams

$$
\begin{CD}
B_\beta @. @. @. B_\beta\\
@AAA \\
B_\alpha @. @. \ \ \ \ \ \ \ \ \equiv \ \ \ \ \ \ \ @. @AAA\\
@AAA  \\
R @>>> S @. @. R_\beta @>>> S_\beta
\end{CD}
$$

This algebra has clearly property 3 of the introduction, so we have the following fact:

\begin{prop}\label{exists}
There exists a tightly $\sigma$-filtered Boolean algebra of size $\mathfrak{c}$ that has the countable separation property.
\end{prop}

For further properties of this algebra we can refer to \cite{DowHart,AviBre}.

\section{The additive $\sigma$-skeleton}

 When working with tightly $\sigma$-filtered algebras, the tower of subalgebras of the definition is often not rich enough, so one needs to consider a larger subfamily of subalgebras. From this moment on, we fix a tightly $\sigma$-filtered Boolean algebra $\mathfrak{B} = \bigcup_{\alpha<\xi}B_\alpha$ and its tower of subalgebras as in the definition. For $\Gamma\subset \xi$ we call $E(\Gamma) = \langle S_i : i\in \Gamma\rangle$. We say that $\Gamma$ is saturated if $R_\gamma \subset E(\Gamma\cap \gamma)$ for every $\gamma\in \Gamma$. The subalgebras of $\mathfrak{B}$ of the form $E(\Gamma)$ with $\Gamma$ saturated constitute a so-called an additive $\sigma$-skeleton, cf.\cite{Geschke,AviBre}. In this section, we check the properties that we need of this family of subalgebras.

\begin{lem}\label{tightness}
For every countable set $H\subset \mathfrak{B}$ there exists a countable saturated set $\Gamma_H\subset\xi$ such that $H\subset E(\Gamma_H)$.
\end{lem}

\begin{proof}
It is clear that we can find a countable set $\Gamma$ such that $H\subset E(\Gamma)$. So what we need to check is that there exists a larger countable $\Gamma_H\supset \Gamma$ which is saturated. For every $\gamma\in \xi$, let $\Delta_\gamma$ be a countable subset of $\gamma$ such that $R_\gamma\subset E(\Delta_\gamma)$. For every set $Z\subset \xi$, let $\Delta' = \bigcup_{\gamma\in Z}\Delta_\gamma$. We can take $\Gamma_H = \Gamma\cup \Gamma'\cup\Gamma''\cup\cdots$.
\end{proof}

\begin{lem}\label{pushoutlemma}
Let $\Gamma_1$ and $\Gamma_2$ be subsets of $\xi$ such that $\Gamma_1$, $\Gamma_2$ and $\Gamma_1\cap \Gamma_2$ are all saturated, then we have a push-out diagram of inclusions
$$
\begin{CD}
E(\Gamma_1)@>>> E(\Gamma_1\cup \Gamma_2)\\
@AAA @AAA\\
E(\Gamma_1\cap\Gamma_2) @>>> E(\Gamma_2).
\end{CD}
$$
\end{lem}

\begin{proof}
We must proof that if $x\in E(\Gamma_1)$ and $y\in E(\Gamma_2)$ are such that $x\wedge y = 0$, then there exist $r,s\in E(\Gamma_1\cap \Gamma_2)$ such that $x\leq r$, $y\leq s$ and $r\wedge s = 0$. Notice that this already implies that $E(\Gamma_1\cap \Gamma_2) = E(\Gamma_1)\cap E(\Gamma_2)$, by taking $x$ and its complement $y=\bar{x}$ in the intersection. It is obvious that $\langle E(\Gamma_1)\cup E(\Gamma_2)\rangle = E(\Gamma_1\cup \Gamma_2).$ Let
$$\zeta' = \min \left\{\alpha : x\in E(\Gamma_1\cap \alpha) \text{ and } y\in E(\Gamma_2\cap \alpha)\right\}.$$

We prove the lemma by induction on $\zeta'$. Since every element in $E(\Gamma)$ must belong to $E(F)$ for some finite $F\subset \Gamma$, we see that $\zeta' = \zeta+1$ must be a successor ordinal. We can write $x$ and $y$ in the form
$$ x = \bigvee_{i<n} v^1_i \wedge w^1_i$$ $$y = \bigvee_{j<m} v^2_j \wedge w^2_j$$
with $v^k_i \in E(\Gamma_k \cap \zeta)$ and $w^k_i \in S_\zeta$. We must have that either $x\not\in E(\Gamma_1\cap \zeta)$ or $y\not\in E(\Gamma_2\cap \zeta)$. Withouth loss of generality, we assume that $y\not\in E(\Gamma_2\cap \zeta)$, which leaves two cases:

Case 1: When $x\not\in E(\Gamma_1\cap \zeta)$ and $y\not\in E(\Gamma_2\cap \zeta)$. This implies that $\zeta\in\Gamma_1\cap\Gamma_2$. Since $x\wedge y = 0$ we have that $v_i^1\wedge w_i^1 \wedge v_j^2 \wedge w_j^2 = 0$ for every $i,j$. We have that $v_i^1\wedge v_j^2\in B_\zeta$ and $w^1_i\wedge w^2_j\in S_\zeta$. Since there is a push-out diagram of $B_\zeta$, $S_\zeta$ and $R_\zeta = B_\zeta\cap S_\zeta$ from the definition of tightly $\sigma$-filtered, there exists $r_{ij}\in R_\zeta$ such that $r_{ij}\geq w_i^1\wedge w_j^2$ and $v^1_i\wedge v^2_j\wedge r_{ij} = 0$. Since $\Gamma_2$ is saturated and $\zeta\in \Gamma_2$, we have $R_\zeta\subset E(\Gamma_2\cap \zeta)$. Therefore $v^1_i\in E(\Gamma_1\cap\zeta)$ and $ v^2_j\wedge r_{ij}\in E(\Gamma_2\cap\zeta)$ and $v^1_i\wedge v^2_j\wedge r_{ij} = 0$. So can apply the inductive hypothesis, and we find $u^1_{ij},u^2_{ij}\in E(\Gamma_1\cap\Gamma_2)$ such that $u^1_{ij}\geq v^1_i$, $u^2_{ij}\geq v^2_j\wedge r_{ij}$ and $u^1_{ij}\cap u^2_{ij} = 0$. Consider now
$$ x' = \bigvee_{i<n} \bigwedge_{j<m} u^1_{ij} \wedge w^1_i$$ $$y' = \bigvee_{j<m} \bigwedge_{i<n} (u^2_{ij}\vee \overline{r_{ij}}) \wedge w^2_j$$
We have that $x'\geq x$, $y'\geq y$ and $x'\wedge y'  = 0$. Moreover $u_{ij}^k\in E(\Gamma_1\cap \Gamma_2)$, $w^k_i\in S_\zeta$ and $r_{ij}\in R_\zeta\subset S_\zeta$. Since in this case $\zeta\in \Gamma_1\cap \Gamma_2$, we have that $S_\zeta\subset E(\Gamma_1\cap\Gamma_2)$. Therefore, $x',y'\in E(\Gamma_1\cap\Gamma_2)$ and we are done.

Case 2:  When $x\in E(\Gamma_1\cap \zeta)$. For every $j<m$ we have that $x\wedge v^2_j\wedge w^2_j = 0$. Now, $x\wedge v^2_j\in B_\zeta$ and $w^2_j\in S_\zeta$, so we can find $r_j,s_j\in R_\zeta$ such that $x\wedge v^2_j\leq r_j$, $w^2_j\leq s_j$ and $r_j\wedge s_j = 0$. Consider $y' = \bigvee_{j<m} v^2_j\wedge s_j$. On the one hand, since $\Gamma_2$ is saturated, we have that $s_j\in R_\zeta\subset E(\Gamma_2\cap \zeta)$, so $y'\in E(\Gamma_2\cap \zeta)$. On the other hand, $y\leq y'$ and $x\wedge y'=0$. We can apply the inductive hypothesis and we are done. 
 
\end{proof}

\begin{lem}\label{artihmetics}
Let $P(x_1,\ldots,x_n)$ be a Boolean polynomial on $n$ variables. Then,
\begin{enumerate}
\item There exist two Boolean polynomials $P^-$ and $P^+$ on $n-1$ variables such that the equation $P(x_1,\ldots,x_n)=0$ is equivalent to
$$P^-(x_1,\ldots,x_{n-1})\leq x_n \leq P^+(x_1,\ldots,x_{n-1})$$
\item Suppose that we have elements in a Boolean algebra $r_i^-\leq a_i \leq r_i^+$ for $i=1,\ldots,n$. If $P(r^{\varepsilon_1}_1,\ldots,r^{\varepsilon_n}_n) =  0$ for any choice of signs $\varepsilon_i\in \{+,-\}$, then $P(a_1,\ldots,a_n)=0$.
\end{enumerate} 
\end{lem}

\begin{proof}
Elementary arithmetics of Boolean algebras.
\end{proof}

\begin{lem}\label{multipushout}
Let $\Delta,\Gamma_1,\ldots,\Gamma_n$ be saturated subsets of $\xi$ such that $\Gamma_i\cap \Gamma_j = \Delta$ for all $i\neq j$. Then, whenever we have elements $a_i\in E(\Gamma_i)$ that satisfy a Boolean equation $P(a_1,\ldots,a_n)=0$, then there exist elements $r^-_i\leq a_i \leq r^+_i$ such that $r^-_i,r^+_i\in E(\Delta)$ and $P(r^{\varepsilon_1}_1,\ldots,r^{\varepsilon_n}_n) =  0$ for any choice of signs $\varepsilon_i\in \{+,-\}$. 
\end{lem}

\begin{proof}
For every $k\in\{1,\ldots,n\}$ we can find two Boolean polynomials $P_{k}^-$ and $P_k^+$ on $n-1$ variables, such that the equation $P(x_1,\ldots,x_n) = 0 $ can be equivalently rewriten as
$$P^-_k(x_1,\ldots,x_{k-1},x_{k+1},\ldots,x_n) \leq x_k \leq P^+_k(x_1,\ldots,x_{k-1},x_{k+1},\ldots,x_n).$$
We choose the elements $r_k^+,r_k^-\in E(\Delta)$ by induction on $k$ with the property that $r_k^-\leq a_k\leq r_k^+$ and $P(r^{\varepsilon_1}_1,\ldots,r^{\varepsilon_k}_k,a_{k+1},\ldots,a_n) =  0$ for any choice of signs $\varepsilon_i$. Suppose that we have already chosen $r_1^-,r_1^+,\ldots,r_{k-1}^-,r_{k-1}^+$. Then we have
$$\bigvee_{\varepsilon_i\in \{+,-\}} P^-_k(r^{\varepsilon_1}_1,\ldots,r^{\varepsilon_{k-1}}_{k-1},a_{k+1},\ldots,a_n)  $$ $$ \leq a_k \leq $$ $$ \bigwedge_{\varepsilon_i\in\{+,-\}} P^+_k(r^{\varepsilon_1}_1,\ldots,r^{\varepsilon_{k-1}}_{k-1},a_{k+1},\ldots,a_n)$$
The lower and upper bounds in this chain of inequalities belong to $E(\Gamma_{k+1}\cup\cdots\cup \Gamma_n)$ while the middle term $a_k$ belongs to $E(\Gamma_k)$. Since $\Gamma_k \cap (\Gamma_{k+1}\cup\cdots\cup \Gamma_n) = \Delta$, we can apply Lemma~\ref{pushoutlemma} that implies the existence of $r_k^-,r_k^+\in E(\Delta)$ such that
$$\bigvee_{\varepsilon_i\in \{+,-\}} P^-_k(r^{\varepsilon_1}_1,\ldots,r^{\varepsilon_{k-1}}_{k-1},a_{k+1},\ldots,a_n) \leq $$ $$r_k^-\leq a_k\leq r_k^+$$ $$ \leq \bigwedge_{\varepsilon_i\in\{+,-\}} P^+_k(r^{\varepsilon_1}_1,\ldots,r^{\varepsilon_{k-1}}_{k-1},a_{k+1},\ldots,a_n).$$
\end{proof}

\section{A principle of symmetry in $\mathfrak{B}$}

Remember the standard notation that $[A]^\kappa$ denotes the family of all subsets of $A$ of cardinality $\kappa$. 

\begin{lem}\label{Deltasystem}
Let $\mathcal{F}$ be a family of countable sets of cardinality $\omega_2$. Then there exists  $\mathcal{G}\in [\mathcal{F}]^{\omega_2}$  and a countable set $\Delta$ such that the sets $\{A\setminus \Delta : A\in \mathcal{G}\}$ are pairwise disjoint.
\end{lem}

\begin{proof}
Enumerate $\mathcal{F} = \{F_\alpha : \alpha<\omega_2\}$. Notice that $|\bigcup\mathcal{F}|\leq \omega\cdot\omega_2 = \omega_2$. Suppose first that $|\bigcup \mathcal{F}|<\omega_2$. Then we can suppose that all sets in $\mathcal{F}$ are subsets of the ordinal $\omega_1$ and we can write $\mathcal{F} = \bigcup_{\gamma<\omega_1} \{F\in\mathcal{F} : F\subset \gamma\}$. Since $|\mathcal{F}|=\omega_2$, there would exist $\gamma$ such that $|\mathcal{F}_\gamma| = \omega_2$, and then we can take $\mathcal{G} = \mathcal{F}_\gamma$ and $\Delta=\gamma$. This proves the lemma in the case when $|\bigcup\mathcal{F}|<\omega_2$, so we will suppose now that $|\bigcup\mathcal{G}|=\omega_2$ for all subfamilies $\mathcal{G}\in [\mathcal{F}]^{\omega_2}$. We also suppose that all sets in $\mathcal{F}$ are subsets of the ordinal $\omega_2$. For every $\alpha<\omega_2$ of cofinality $\omega_1$ choose an ordinal $\zeta(\alpha)<\alpha$ that is an upper bound of $F_\alpha\cap \alpha$. Applying the pressing down lemma, we obtain $\zeta<\omega_2$ and a set $A\in [ \omega_2]^{\omega_2}$ such that $\zeta(\alpha) = \zeta$ for all $\alpha\in A$. This implies that $F_\alpha\cap F_\beta \subset \zeta$ whenever $F_\alpha\subset \beta$. Since we are assuming that $|\bigcup\{F_\alpha : \alpha\in A\}|=\omega_2$, it is possible to construct by induction a subset $B\subset A$ of cardinality $\omega_2$ such that $F_\alpha\subset \beta$ for all $\alpha,\beta\in B$ with $\alpha<\beta$. With this we get $\{F_\alpha\setminus \zeta : \alpha\in A\}$ are pairwise disjoint. On the other hand, we already proved the case when the union has cardinality less than $\omega_2$, so we can find $C\subset B$ of cardinality $\omega_2$ and a countable $\Delta$ such that  $\{(F_\alpha\cap\zeta)\setminus\Delta : \alpha\in C\}$ are pairwise disjoint. The set $\Delta$ and the family $\mathcal{G} = \{F_\alpha : \alpha\in C\}$ were those that we were looking for.
\end{proof}

\begin{thm}\label{subfamilyP}
Let $\mathcal{H}$ be a family of countable subsets of $\mathfrak{B}$ of cardinality $\omega_2$. Then, there exists $\mathcal{H'}\in [\mathcal{H}]^{\omega_2}$  and a countable subalgebra $R\subset \mathfrak{B}$ such that for every different $H_1,\ldots,H_n\in\mathcal{H}'$ and any $a_i\in H_i$, if a Boolean equation $P(a_1,\ldots,a_n)=0$ is satisfied, then there exist $r^-_i,r^+_i\in R$ such that $r^-_i\leq a_i\leq r^+_i$ and $P(r^{\varepsilon_1}_1,\ldots,r^{\varepsilon_n}_n) =  0$ for any choice of signs $\varepsilon_i\in \{+,-\}$.
\end{thm}

\begin{proof}
Using Lemma~\ref{tightness}, for every $H\in \mathcal{H}$ choose a saturated countable set $\Gamma_H\subset \xi$ such that $H\subset E(\Gamma_H)$. Apply Lemma~\ref{Deltasystem} to pass to a subfamily of cardinality $\omega_2$ with a countable set $\Delta\subset\xi$ as in that lemma. By enlarging $\Delta$ and each $\Gamma_H$, we can suppose that $\Delta$ is saturated and $\Delta\subset \Gamma_H$. Then we apply Lemma~\ref{multipushout}.
\end{proof}

If $R$ is a Boolean algebra, we consider $G(R)$ to be the set of pairs $$ G(R) = \left\{ (I^-,I^+)\in 2^R \times 2^R : I^- \text{ is an ideal}, I^+ \text{ is a filter}, I^-\leq I^+\right\}.$$ 

Here $I^-\leq I^+$ means that $a\leq b$ whenever $a\in I^-$ and $b\in I^+$. The set $2^R$ is the family of all subsets of $R$, that is identified with the power space $\{0,1\}^R$ endowed with the product topology. When $R$ is countable, $2^R$ is homeomorphic to the Cantor set. Notice that $G(R)$ is a closed subset of $2^R\times 2^R$, so it is a compact space, which is moreover metrizable when $R$ is countable. 

Given a Boolean algebra $R$ and a Boolean polynomial $P(x_1,\ldots,x_n)$ on $n$ variables, we define the set $G_P(R)\subset G(R)^n$ as the set all $((I^-_1,I^+_1),\ldots,(I^-_n,I^+_n))\in G(R)^n$ such that there exist $r_1^-,r_1^+,\ldots,r_n^-,r_n^+$  such that $r_i^\varepsilon\in I_i^\varepsilon$ for all $i,\varepsilon$ and $P(r^{\varepsilon_1}_1,\ldots,r^{\varepsilon_n}_n) =  0$ for any choice of signs $\varepsilon_i\in \{+,-\}$. More generally, if $W$ is a set of coordinates and $w_1,\ldots,w_n\in W$, we define the set $G_P^W[w_1,\ldots,w_n](R)\subset G(R)^W$ as the set all 
$$((I^-_w(1),I^+_w(1)),\ldots,(I^-_w(n),I^+_w(n)))_{w\in W}\in (G(R)^W)^n$$
 such that there exist $r_1^-,r_1^+,\ldots,r_n^-,r_n^+$  such that $r_i^\varepsilon\in I_{w_i}^\varepsilon(i)$ for all $i,\varepsilon$ and also $P(r^{\varepsilon_1}_1,\ldots,r^{\varepsilon_n}_n) =  0$ for any choice of signs $\varepsilon_i\in \{+,-\}$. The following is an easy observation:

\begin{lem}
The set $G_P^W[w_1,\ldots,w_n](R)$ is an open subset of $(G(R)^W)^n$, and the set $G_P(R)$ is an open subset of $G(R)^n$.
\end{lem}

\begin{thm}\label{omegathm}
Let $\mathcal{H}$ be a family of countable subsets of $\mathfrak{B}$ of cardinality $\omega_2$, and let $P(x_1,\ldots,x_n)$ be a Boolean polynomial. Then,
\begin{enumerate}
\item either there exists $\mathcal{H}_0 \in [\mathcal{H}]^{\omega_2}$  such that  $P(a_1,\ldots,a_n) \neq 0$ whenever we have different $H_1,\ldots,H_n\in \mathcal{H}_0$ and $a_i\in H_{i}$,
\item or there exist $\mathcal{H}_1,\ldots,\mathcal{H}_n\in [\mathcal{H}]^{\omega_2}$  such that for every $H_1\in\mathcal{H}_1,\ldots,H_n\in\mathcal{H}_n$ there exist $a_i\in H_i$ such that $P(a_1,\ldots,a_n) = 0$.
\end{enumerate}
\end{thm}

\begin{proof}
We first use Theorem~\ref{subfamilyP}, so we can suppose that we have a countable subalgebra $R$ that satisfies the conclusion of that theorem for $\mathcal{H}' = \mathcal{H}$. 
For every $a\in \mathfrak{B}$, let $I^-(a) = \{r\in R : r\leq a\}$, $I^+(a) = \{r\in R: r\geq a\}$. Let $L$ be the compact subset of $G(R)$ obtained by removing those open subsets $\mathcal{U}$ of $G(R)$ such that  $\{H\in\mathcal{H} : \exists a\in H : (I^-(a),I^+(a))\in \mathcal{U}\}$ has cardinality at most $\omega_1$. Since $R$ is countable, $G(R)$ has a countable basis of open sets. There are at most $\omega_1$ many $H\in \mathcal{H}$ for which there exists $a\in H$ such that $(I^-(a),I^+(a))$ lives in one of the removed basic open sets. By removing those $\omega_1$ many $H$ we can also suppose that $(I^-(a),I^+(a))\in L$ for every $a\in \bigcup \mathcal{H}$. If $P(a_1,\ldots,a_n) \neq 0$ whenever we have different $H_1,\ldots,H_n\in \mathcal{H}$ and $a_i\in H_{i}$, then we are done. Otherwise, there exist different sets $H_1,\ldots,H_n\in\mathcal{H}$ and $a_i\in H_i$ such that $P(a_1,\ldots,a_n)=0$. Since we are assuming the conclusion of Theorem~\ref{subfamilyP}, we have then $r_i^\pm \in I^\pm(a_i)$ such that $P(r_1^\pm,\ldots,r_n^\pm)=0$ for all choice of signs. Hence 
$$ ( (I^-(a_1),I^+(a_1)),\ldots,(I^{-}(a_n),I^+(a_n))) \in G_P(R).$$
Since $G_P(R)$ is open, we can find open neighborhoods $U_i$ of $(I^-(a_1),I^+(a_1))$ in $G(R)$ such that $$ (\star)\ \ U_1\times\cdots\times U_n \subset G_P(R).$$
For every $i=1,\ldots,n$, $(I^-(a_i),I^+(a_i))\in L$, which means that there exist $\mathcal{H}_i\in [\mathcal{H}]^{\omega_2}$ such that for each $H\in \mathcal{H}_i$ there exists $a\in H$ such that $(I^-(a),I^+(a))\in U_i$. In that case, the families $\mathcal{H}_i$ satisfy the second case of the theorem, by $(\star)$, the definition of $G_P(R)$, and Lemma~\ref{artihmetics}. \end{proof}

\begin{cor}[Dow, Hart]
$\mathfrak{B}$ does not contain $\omega_2$-chains.
\end{cor}

\begin{proof}
Suppose $\{a_\alpha : \alpha<\omega_2\}$ is an $\omega_2$-chain in $\mathfrak{B}$. Apply Theorem~\ref{omegathm} for $H_\alpha = \{a_\alpha\}$ and $P(x_1,x_2) = x_1\setminus x_2$. Possibility (1) is impossible as it would imply that $a_{\alpha_1} \not< a_{\alpha_2}$ for some $\alpha_1<\alpha_2$. Possiblity (2) is impossible as well as it would imply that $a_{\alpha_1}\subset a_{\alpha_2}$ for some $\alpha_2>\alpha_1$.
\end{proof}

\section{A principiple of symmetry in $C(\mathfrak{K})$}

Now, let $\mathfrak{K}$ be the Stone space of the tightly $\sigma$-filtered Boolean algebra $\mathfrak{B}$, and $C(\mathfrak{K})$ the Banach space of continuous real-valued functions on $\mathfrak{K}$, endowed with the supremum norm. For convenience, we will consider that $\mathfrak{B}$ is the family of clopen subsets of $\mathfrak{K}$.

\begin{lem}\label{norminequality}
Let $\mathcal{F}$ be a family of countable subsets of $C(\mathfrak{K})$ of cardinality $\omega_2$, and let $\lambda_1,\cdots,\lambda_n,M$ be scalars. Suppose that for every $\mathcal{G}\in [\mathcal{F}]^{\omega_2}$ we can find different $F_1,\ldots,F_n\in \mathcal{G}$ and $f_i\in F_i$ such that $$\left\| \sum \lambda_i f_i \right\| < M.$$
Then, there exist $\mathcal{G}_1,\ldots,\mathcal{G}_n\in [\mathcal{F}]^{\omega_2}$ such that for every $G_1\in \mathcal{G}_1,\ldots, G_n\in \mathcal{G}_n$, there exist $g_i\in G_i$ such that
$$\left\| \sum \lambda_i g_i \right\| < M.$$
\end{lem}

\begin{proof}
For every rationals $p<q$ and every $f\in C(\mathfrak{K})$ fix a clopen set $c[f,p,q]\in \mathfrak{B}$ such that
$$ \{x\in \mathfrak{K} : f(x)\leq p\} \subset c[f,p,q] \subset \{x\in\mathfrak{K} : f(x)<q\}.$$
Let us call $W = \{(p,q)\in\mathbb{Q}^2 : p<q\}$.

If we consider $H_F = \{c[f,p,q] : f\in F, (p,q)\in W\}$, then we can apply Theorem~\ref{subfamilyP} to the family $\mathcal{H} = \{H_F : F\in\mathcal{F}\}$. Thus, wihout loss of generality, we shall assume that   we have a countable algebra $R\subset \mathfrak{B}$ and that this family $\mathcal{H}$ satisfies the property stated in Theorem~\ref{subfamilyP} for $\mathcal{H}'$. For every $a\in \mathfrak{B}$, let $I^-(a) = \{r\in R : r\leq a\}$, $I^+(a) = \{r\in R: r\geq a\}$, and $I(a) = (I^-(a),I^+(a))\in G(R)$. For every $f\in C(\mathfrak{K})$, let also $J(f) = (I(c[f,p,q]))_{(p,q)\in W} \in G(R)^W$, and $J[F] = \{J(f) : f\in F\}$. Since $W$ is countable, $G(R)^W$ is a compact metrizable space.

We consider the compact subset $L$ of $G(R)^W$ obtained by removing  all basic open subsets $V$ of $G(R)^W$ for which $|\{F\in \mathcal{F} : V\cap J[F] \neq \emptyset\}| < \omega_2$, and we also remove from $\mathcal{F}$ all $F$ whose $J[F]$ have nonempty intersections with the removed $V$'s. In this way, we can suppose that $J(f)\in L$ for all $f\in F\in \mathcal{F}$, and also that for every nonempty open subset $V$ of $L$,  $$\{F\in \mathcal{F} : J[F]\cap V\neq \emptyset\}|=\omega_2.$$
By the hypothesis of the theorem, we can find  different $F_1,\ldots,F_n\in \mathcal{F}$ and $f_i\in F_i$ such that $$\left\| \sum \lambda_i f_i \right\| < M.$$
Fix also $M'$ such that

$$\left\| \sum \lambda_i f_i \right\| <M' < M,$$

take $\varepsilon<\min(M-M',1)$, and $N\in \mathbb{N}$ such that $\max_i \|f_i\| < N$. The function $\phi:[-N,N]^n \To \mathbb{R}$ given by $\phi(\mu_1,\ldots,\mu_n) = \|\sum_i \lambda_i \mu_i\|$ is uniformly continuous, so we can find $\delta\in (0,\varepsilon)$ such that $\|x-y\|_\infty <4\delta$ implies $|\phi(x)-\phi(y)|<\varepsilon$. We can also take $\delta$ to be the inverse of a natural number. Let $Q_0 = \delta \mathbb{Z} \cap [-5N,5N]$. Thus, $Q_0$ is a finite set of rationals that is $\delta$-dense in $[-5N,5N]$. Given a Boolean polynomial $P(x_1,\ldots,x_n)$ and $(p_i,q_i)\in W\cap Q_0^2$ such that
$$ P(c[f_1,p_1,q_1],\ldots,c[f_n,p_n,q_n]) = 0,$$
what we have is that
$$(J(f_1),\ldots,J(f_n)) \in G^W_P[(p_1,q_1),\ldots,(p_n,q_n)](R)$$
Thus we can find open neighborhoods $V_i$ of $J(f_i)$ in $L$ such that

$$V_1\times\cdots\times V_n \subset$$ 
$$\bigcap\left\{G^W_P[(p_1,q_1),\ldots,(p_n,q_n)](R) : P(c[f_1,p_1,q_1],\ldots,c[f_n,p_n,q_n]) = 0, (p_i,q_i)\in Q_0^2\right\}.$$
By the properties of $L$, the sets
$$\mathcal{G}_i = \{F\in \mathcal{F} : V_i\cap J[F] \neq \emptyset\}$$
have cardinality $\omega_2$. Now, we consider any $G_i\in \mathcal{G}_i$ and we pick $g_i\in G_i$ such that $J(g_i)\in V_i$.
This implies that for every Boolean polynomial $P(x_1,\ldots,x_n)$ and every $(p_i,q_i)\in W\cap Q_0^2$ we have that
$$ P(c[f_1,p_1,q_1],\ldots,c[f_n,p_n,q_n]) = 0 \ \Rightarrow P(c[g_1,p_1,q_1],\ldots,c[g_n,p_n,q_n]) = 0.$$
In particular, $\|g_i\|\leq N+\delta$ for every $i$, because
$$ \|f_i\| < N \Rightarrow c[f_i,N,N+\delta]^c \cup c[f_i,-N-\delta,-N] = 0 $$  $$\Rightarrow c[g_i,N,N+\delta]^c\cup c[g_i,-N-\delta,-N] = 0 \Rightarrow \|g_i\| \leq N+\delta.$$

We claim that
$$\left\| \sum \lambda_i g_i \right\| < M.$$
Suppose for contradicition that this was not the case. This means that there is a point $x\in \mathfrak{K}$ such that 
$$\left|\sum \lambda_i g_i(x)\right| \geq M.$$
choose $p_i\leq g_i(x) \leq q_i$ with $p_i,q_i\in Q_0$ and $q_i-p_i=\delta.$ We have that
$$ x\in \bigcap_i c[g_i,q_i,q_i+\delta] \setminus c[g_i,p_i-\delta,p_i],$$
in particular
$$\bigcap_i c[g_i,q_i,q_i+\delta] \setminus c[g_i,p_i-\delta,p_i] \neq 0, $$
therefore
$$\bigcap_i c[f_i,q_i,q_i+\delta] \setminus c[f_i,p_i-\delta,p_i] \neq 0. $$
If we choose a point $y$ in the nonempty clopen set above, then $p_i-\delta< f_i(y) < q_i+\delta$. Hence $|f_i(y)-g_i(x)|<3\delta$. By the choice of $\delta$ from uniform continuity, this implies that
$$\left|\sum \lambda_i f_i(y)\right| \geq M-\varepsilon > M',$$
a contradiction.
\end{proof}

\begin{thm}\label{notellinfty}
Suppose that  the Boolean algebra $\mathcal{P}(\omega)/fin$ contains an $\omega_2$-chain. Then $\ell_\infty$ is not isomorphic to a subspace of $C(\mathfrak{K})$.
\end{thm}

\begin{proof}
Suppose that $T:\ell_\infty\To C(\mathfrak{K})$ is an isomorphism onto a subspace, and that $\{c_\alpha\}$ is an increasing $\subset^\ast$-chain in $\mathcal{P}(\omega)$. We shall reach a contradiciton. Let $C_\alpha$ be the equivalence class of $c_\alpha$ modulo finite sets. For each $\alpha<\omega_2$, let $F_\alpha = \{T(1_c) : c \in C_\alpha\}$, and let $\mathcal{F} = \{ F_\alpha : \alpha<\omega_2\}$. We choose a natural number $n>\|T\|\cdot \|T^{-1}\|$. For every $\alpha_1<\cdots<\alpha_{2n}$, we can choose $c_i\in C_{\alpha_i}$ such that $c_1\subset c_2\subset \cdots\subset c_n$, and in this way we have
$$\left\| \sum_{i=1}^n 1_{c_{2i}} - 1_{c_{2i-1}}\right\| = 1 $$
and
$$\left\| \sum_{i=1}^n T(1_{c_{2i}}) - T(1_{c_{2i-1}})\right\| \leq \|T\|.$$
We are in a position to apply Lemma~\ref{norminequality}, and we get $\mathcal{G}_1,\ldots,\mathcal{G}_{2n}\in [\omega_2]^{\omega_2}$ such that whenever we pick $\gamma_i\in \mathcal{G}_i$ we can find $c_i\in C_{\gamma_i}$ such that
$$(\star)\ \ \ \left\| \sum_{i=1}^n T(1_{c_{2i}}) - T(1_{c_{2i-1}})\right\| \leq \|T\|.$$
Since the sets $\mathcal{G}_i$ are cofinal in $\omega_2$ we can pick such ordinals $\gamma_i$ to follow the order $$\gamma_1<\gamma_3<\gamma_5<\cdots<\gamma_{2n-1}<\gamma_{2n}<\gamma_{2n-2}<\cdots<\gamma_4<\gamma_2.$$
Then, 
$$c_1\subset^\ast c_3\subset^\ast c_5\subset^\ast \cdots\subset^\ast c_{2n-1}\subset^\ast c_{2n}\subset^\ast c_{2n-2}\subset^\ast \cdots\subset^\ast c_4\subset^\ast c_2,$$
hence $$\bigcap_{i=1}^n c_{2i}\setminus c_{2i-1} \neq \emptyset,$$
$$\left\| \sum_{i=1}^n 1_{c_{2i}} - 1_{c_{2i-1}}\right\| = n,$$
$$ \left\| \sum_{i=1}^n T(1_{c_{2i}}) - T(1_{c_{2i-1}})\right\| \geq n\|T^{-1}\|^{-1}.$$
But this contradicts the inequality $(\star)$ and our choice of  $n>\|T\|\cdot \|T^{-1}\|$.
\end{proof}

If we apply Theorem~\ref{notellinfty} when $\mathfrak{K}$ is the Stone space of the Boolean algebra of Proposition~\ref{exists}, we get an F-space $\mathfrak{K}$, so that $C(\mathfrak{K})$ is 1-separably injective:

\begin{cor}\label{main-result}
Suppose that  the Boolean algebra $\mathcal{P}(\omega)/fin$ contains an $\omega_2$-chain. Then there exists a 1-separably injective Banach space $C(\mathfrak{K})$ such that $\ell_\infty$ is not isomorphic to a subspace of $C(\mathfrak{K})$.
\end{cor}

The hypothesis that $\mathcal{P}(\omega)/fin$ contains an $\omega_2$-chain holds under Martin's Axiom and the negation of the Continuum Hypothesis.

\end{document}